\documentclass[a4paper,11pt]{article}
\usepackage[T1]{fontenc}
\usepackage{lmodern}
\usepackage[margin=25mm]{geometry}
\usepackage{amsmath,amssymb,amsthm,needspace}
\usepackage{microtype}
\usepackage[hidelinks]{hyperref}

\theoremstyle{definition}
\newtheorem{definition}{Definition}[section]
\theoremstyle{plain}
\newtheorem{lemma}[definition]{Lemma}
\newtheorem{theorem}[definition]{Theorem}
\newtheorem{corollary}[definition]{Corollary}
\newcommand{\Sym}{\operatorname{Sym}}
\newcommand{\cl}{\operatorname{cl}}
\hypersetup{pdftitle={The Vamos Matroid Has No Second Symmetric Power},pdfauthor={Yasuhito Nakajima}}
\title{The V\'amos Matroid Has No Second Symmetric Power}
\author{Yasuhito Nakajima}

\date{}

\begin{document}
\maketitle

\begin{abstract}
Symmetric powers of matroids were introduced by Lov\'asz and Mason.
The subject has recently attracted renewed interest through its connection
with the realizability of tropical linear spaces by tropical ideals.
Anderson proved that the tropical linear space associated with a matroid
$M$ is the variety of a tropical ideal if and only if $M$ has a
$d$th symmetric power for every positive integer $d$.
He also asked whether the V\'amos matroid $V_8$ has a second symmetric
power. We answer this question in the negative by proving that $V_8$ has
no second symmetric power. By Anderson's characterization, the associated
tropical linear space cannot be realized as the variety of a tropical ideal.
\end{abstract}

\section{Introduction}\label{sec:introduction}

Symmetric powers of matroids were introduced by Lov\'asz and
Mason~\cite{Lovasz,Mason}. 
More recently,
connections with tropical geometry have provided new motivation for
studying the existence of symmetric powers.
Maclagan and Rinc\'on introduced tropical ideals, which include
tropicalizations of classical ideals~\cite{MaclaganRincon}.
They subsequently proved that the top-dimensional part of the variety of a tropical ideal, equipped with suitable weights, is a balanced polyhedral complex~\cite[Theorem 1.2]{MaclaganRinconBalanced}.
The Bergman fan of a matroid satisfies the balancing condition when each
maximal cone is assigned weight one. Throughout this paper,
we call this weighted fan the tropical linear space associated
with the matroid. When comparing such a tropical linear space with
the variety of a tropical ideal, we require equality of the weighted
polyhedral complexes, after passing to a common refinement if necessary.
Draisma and Rinc\'on connected tropical ideals with matroid products and
proved that the tropical linear space associated with
$V_8\oplus U_{2,3}$, the direct sum of the V\'amos matroid and the
rank-two uniform matroid on three elements, is not the variety of a
tropical ideal~\cite[Theorem~5.2]{DraismaRincon}.
They left open whether the tropical linear space associated with $V_8$
 can be realized as the variety of a tropical ideal.

Anderson clarified the relationship between realizability by tropical
ideals and the existence of symmetric powers~\cite{Anderson}.
A $d$th symmetric power of a matroid $M$ with ground set $E(M)$ is a
matroid on $\Sym^d(E(M))$ satisfying compatibility conditions with
symmetric powers of lower degree. Anderson proved that the associated
tropical linear space $\operatorname{trop}(M)$ is the variety of a
tropical ideal if and only
if $M$ has a $d$th symmetric power for every positive integer
$d$~\cite[Theorem~1.4]{Anderson}.
Thus, nonexistence of a symmetric power in a single degree implies
that the corresponding tropical linear space cannot be realized as the
variety of a tropical ideal.
In this context, Anderson asked whether the V\'amos matroid has a
second symmetric power~\cite[Question~2.28]{Anderson}.

\Needspace{5\baselineskip}
Our main result gives a negative answer to this question.
\begin{theorem}\label{thm:main}
The V\'amos matroid $V_8$ has no second symmetric power.
\end{theorem}

Anderson's characterization immediately implies the following.
\begin{corollary}\label{cor:tropical}
The tropical linear space associated with
the V\'amos matroid cannot be realized as the variety of any tropical ideal.
\end{corollary}

We prove the main theorem by contradiction. Assuming that a second
symmetric power of $V_8$ exists, we obtain rank constraints from its
definition. These constraints yield incompatible rank bounds
for a particular subset $X$ of $\Sym^2(E(V_8))$.

The paper is organized as follows.
Section~\ref{sec:preliminaries} recalls the matroid terminology used in the
proof. Section~\ref{sec:main-proof} contains the proof of the main result.
In that section, we first describe the V\'amos matroid, introduce product notation on
$\Sym^2(E)$, and define second symmetric powers. We then establish the
required rank bounds and apply them to the four-block configuration of
$V_8$ to prove Theorem~\ref{thm:main}.

\paragraph{Use of AI.}
After proving the preliminary results up to and including
Lemma 3.5, the author used ChatGPT 6 Astra to search for subsets
that would yield a contradiction in the proof of the main theorem.
The author also used suggestions from this model when formulating the
statements and proofs from Lemma 3.6 onward.
The author independently checked the mathematical content of these
suggestions and revised the arguments as necessary.
ChatGPT 6 Astra was also used to translate the manuscript from Japanese
into English. ChatGPT 6 Astra and ChatGPT 6 Pro were used  to improve the clarity of the English exposition.
The author has verified all definitions, statements, proofs, and
conclusions in this paper and takes full responsibility for the entire
content of the paper, including its mathematical correctness.

\section{Matroid terminology and basic properties}\label{sec:preliminaries}
For standard terminology and basic results on matroids, we refer to
Oxley~\cite{Oxley}.
All sets considered in this paper are finite.
We write $2^E$ for the set of all subsets of $E$ and $|Y|$ for the
cardinality of a set $Y$.

\begin{definition}
A \emph{matroid} $M$ on a set $E$ is specified by an integer-valued
function $\rho_M:2^E\to\mathbb Z_{\ge0}$ satisfying
\begin{align*}
0\le\rho_M(Y)&\le |Y| &&(Y\subseteq E),\\
\rho_M(Y)&\le\rho_M(Z) &&(Y\subseteq Z\subseteq E),\\
\rho_M(Y)+\rho_M(Z)&\ge
\rho_M(Y\cup Z)+\rho_M(Y\cap Z) &&(Y,Z\subseteq E).
\end{align*}
The set $E$ is the \emph{ground set}, the function $\rho_M$ is the
\emph{rank function}, and $\rho_M(Y)$ is the \emph{rank} of $Y$.
The rank of $M$ is $\rho_M(E)$.
The second and third axioms are called \emph{monotonicity} and
\emph{submodularity}, respectively.
\end{definition}

\begin{definition}
A subset $I\subseteq E$ is \emph{independent} if $\rho_M(I)=|I|$,
and \emph{dependent} otherwise.
A maximal independent set is a \emph{basis} of $M$.
Every basis has cardinality $\rho_M(E)$.
A minimal dependent set is a \emph{circuit}; equivalently,
a circuit is dependent and all its proper subsets are independent.
\end{definition}

\begin{definition}
The \emph{closure} of $Y\subseteq E$ is
\[
\cl_M(Y):=\{e\in E:\rho_M(Y\cup\{e\})=\rho_M(Y)\}.
\]
We say that $Y$ \emph{spans} $Z$ if $Z\subseteq\cl_M(Y)$.
In particular, $Y$ spans $M$ if $\cl_M(Y)=E$.
A subset $F$ is a \emph{flat} if $F=\cl_M(F)$.
A flat of rank $\rho_M(E)-1$ is a \emph{hyperplane}.
A set that is both a circuit and a hyperplane is a \emph{circuit-hyperplane}.
\end{definition}

\begin{definition}
For $Y\subseteq E$, the \emph{restriction} of $M$ to $Y$, denoted
$M|_Y$, is the matroid on $Y$ whose rank function is
$Z\mapsto\rho_M(Z)$ for $Z\subseteq Y$.
Its bases are the independent subsets of $Y$ that span $Y$.
An \emph{isomorphism} of matroids is a bijection between their ground
sets that preserves the ranks of all subsets.

An element $e$ is a \emph{loop} if $\rho_M(\{e\})=0$.
Distinct nonloops $e,f$ are \emph{parallel} if
$\rho_M(\{e,f\})=1$.
A matroid is \emph{simple} if it has no loops and no pair of parallel elements.
\end{definition}

The rank axioms imply the following properties of closure:
\[
\begin{gathered}
Y\subseteq\cl_M(Y),\qquad
Y\subseteq Z\ \Longrightarrow\ \cl_M(Y)\subseteq\cl_M(Z),\\
\cl_M(\cl_M(Y))=\cl_M(Y),\qquad
\rho_M(\cl_M(Y))=\rho_M(Y).
\end{gathered}
\]
Every independent set of a matroid extends to a basis.
In particular, if $I\subseteq Y$ is independent and spans $Y$, then
$\rho_M(Y)=|I|$.

\section{Proof of the main result}\label{sec:main-proof}
\subsection{The V\'amos matroid and product notation}\label{sec:vamos}

\begin{definition}
The following presentation agrees, up to relabeling, with that in
Oxley~\cite[Example~2.1.25]{Oxley}.
Let
\[
E=\{1,2,3,4,5,6,7,8\},\qquad
A=\{1,2\},\ B=\{3,4\},\ C=\{5,6\},\ D=\{7,8\},
\]
and set
\[
\mathcal H=\{A\cup B,\ A\cup C,\ A\cup D,\ B\cup C,\ B\cup D\}.
\]
The \emph{V\'amos matroid} $V_8$ is the rank-four matroid on $E$
whose bases are all four-element subsets of $E$ except the five sets
in $\mathcal H$. Its rank function is
\[
\rho_{V_8}(Y)=
\begin{cases}
3,&Y\in\mathcal H,\\
\min\{|Y|,4\},&Y\notin\mathcal H.
\end{cases}
\]
\end{definition}

The matroid $V_8$ is simple, and every three-element subset is independent.
Each member of $\mathcal H$ is a circuit-hyperplane, whereas $C\cup D$
is a basis. The union of any three or more blocks, such as $A\cup B\cup D$,
has rank four. 
For the remainder of the paper, let
$E=\{1,2,\ldots,8\}$
be the ground set of $M=V_8$.

\begin{definition}
The \emph{second symmetric power} of a set $E$, denoted $\Sym^2(E)$,
is the set of unordered pairs of elements of $E$, with repetition allowed.
We denote these pairs by $ij$, with $ij=ji$, and allow $i=j$.
In particular, each diagonal pair $ii$ is an element of $\Sym^2(E)$.
Thus, in our setting,
\[
\Sym^2(E)=\{ij:1\le i\le j\le8\},\qquad
|\Sym^2(E)|=\binom92=36.
\]
For $Y,Z\subseteq E$, put
$YZ:=\{yz:y\in Y,\ z\in Z\}\subseteq\Sym^2(E)$.
Repeated occurrences of the same pair are counted only once.
In particular, $YY=\Sym^2(Y)$.
We abbreviate $\{i\}Y$ to $iY$.
In particular, $iE=\{ij:j\in E\}$.
For $Y\subseteq E$, we write
\[
F_Y:=YE=\bigcup_{i\in Y}iE.
\]
\end{definition}

For example, $AB=\{13,14,23,24\}$ is a subset of $\Sym^2(E)$.
Also, $\Sym^2(D)=\{77,78,88\}$.

\subsection{Second symmetric powers of matroids}\label{sec:symmetric-power}

\begin{definition}
Let $L$ be a loopless matroid with ground set $E_L$ and rank function
$\rho_L$. A matroid $N$ on $\Sym^2(E_L)$ is a \emph{second symmetric
quasi-power} of $L$ if, for every $i\in E_L$, the map
\[
E_L\longrightarrow iE_L,\qquad j\longmapsto ij
\]
induces an isomorphism $L\cong N|_{iE_L}$.
If, in addition,
\[
\rho_N(\Sym^2(E_L))=\binom{\rho_L(E_L)+1}{2},
\]
then $N$ is a \emph{second symmetric power} of $L$.
We consider only loopless matroids in this paper.
\end{definition}

These are Anderson's definitions of symmetric quasi-powers and
symmetric powers~\cite[Definition~1.1]{Anderson}, specialized to $d=2$
and the loopless case. For $d=2$, the nontrivial multiplication maps in that
definition are precisely $j\mapsto ij$ for $i\in E_L$.
Requiring each of them to induce $L\cong N|_{iE_L}$ gives the
quasi-power condition above, and the symmetric-power rank condition
becomes $\rho_N(\Sym^2(E_L))=\binom{\rho_L(E_L)+1}{2}$.

The isomorphisms must be induced by these specified maps; abstract
isomorphisms alone do not suffice.
Consequently, if $N$ is a second symmetric power of $M=V_8$, then
\begin{equation}
\rho_N(iY)=\rho_M(Y)\quad(i\in E,\ Y\subseteq E),
\qquad \rho_N(\Sym^2(E))=10.
\label{eq:defpower}
\end{equation}

Anderson's Question~2.28 asks whether $V_8$ has a second symmetric
power~\cite[Question~2.28]{Anderson}.
We prove Theorem~\ref{thm:main} by deriving the necessary rank conditions
from the basis property of symmetric powers and the isomorphisms
$M\cong N|_{iE}$.

\subsection{Rank lemmas}\label{sec:rank-lemmas}
From this subsection onward, suppose that $N$ is a second symmetric power of $M$. 
The rank identities in \eqref{eq:defpower} imply
\begin{equation}
x\in\cl_M(Y)\quad\Longrightarrow\quad ix\in\cl_N(iY).
\label{eq:transfer}
\end{equation}
Indeed, the hypothesis means that $\rho_M(Y\cup\{x\})=\rho_M(Y)$,
so \eqref{eq:defpower} gives
$\rho_N(iY\cup\{ix\})=\rho_N(iY)$.

The next lemma derives two rank formulas from the basis property of
symmetric powers and the defining rank identities \eqref{eq:defpower}.

\begin{lemma}\label{lem:ranks}
Let $Y\subseteq E$. Then
\[
\rho_N(\Sym^2(Y))=\binom{\rho_M(Y)+1}{2},\qquad
\rho_N(F_Y)=4\rho_M(Y)-\binom{\rho_M(Y)}{2}.
\]
\end{lemma}
\begin{proof}
Choose a basis $I$ of $M|_Y$ and extend it to a basis $K$ of $M$.
By Anderson~\cite[Proposition~2.25]{Anderson}, $\Sym^2(K)$ is a basis
of $N$. Hence its subsets $\Sym^2(I)$ and $IK$ are independent.
Since $Y\subseteq\cl_M(I)$ and $E=\cl_M(K)$,
\eqref{eq:transfer}, together with monotonicity and idempotence of closure,
gives
\[
\begin{aligned}
\Sym^2(Y)&\subseteq\cl_N(IY)\subseteq\cl_N(\Sym^2(I)),\\
F_Y&\subseteq\cl_N(IE)\subseteq\cl_N(IK).
\end{aligned}
\]
Thus $\Sym^2(I)$ spans $\Sym^2(Y)$, and $IK$ spans $F_Y$.
Since $|I|=\rho_M(Y)$, $|K|=4$, and $IK$ is the disjoint union of
$\Sym^2(I)$ and $I(K\setminus I)$, we obtain
\begin{align*}
\rho_N(\Sym^2(Y))&=|\Sym^2(I)|=\binom{\rho_M(Y)+1}{2},\\
\rho_N(F_Y)&=|IK|=\binom{\rho_M(Y)+1}{2}+\rho_M(Y)(4-\rho_M(Y))=4\rho_M(Y)-\binom{\rho_M(Y)}{2}.
\end{align*}
\end{proof}

\begin{lemma}\label{lem:rows}
For distinct $x,y\in E$ and any $Y\subseteq E$,
\begin{equation}
\rho_N(xY\cup yY)\ge
\rho_M(Y)+\rho_M(Y\cup\{x\})-1.
\label{eq:rows}
\end{equation}
\end{lemma}
\begin{proof}
We first establish the auxiliary equality
\[
\rho_N(xE\cup yY)=3+\rho_M(Y\cup\{x\}).
\]
Since $M$ is simple and $x\ne y$, 
applying Lemma~\ref{lem:ranks} to $\{x\}$, $\{y\}$, and $\{x,y\}$ gives
\[
\rho_N(xE)=\rho_N(yE)=4,\qquad
\rho_N(F_{\{x, y\}})=7.
\]
Because  $xE\cap yE=\{xy\}$ and $yY\subseteq yE$,
\[
(xE\cup yY)\cap yE=y(Y\cup\{x\}),\qquad
(xE\cup yY)\cup yE=F_{\{x, y\}}.
\]
Submodularity and \eqref{eq:defpower} therefore imply
\begin{align*}
\rho_N(xE\cup yY)+4
&\ge \rho_N(F_{\{x, y\}})+\rho_N\bigl(y(Y\cup\{x\})\bigr)\\
&=7+\rho_M(Y\cup\{x\}).
\end{align*}
Hence $\rho_N(xE\cup yY)\ge3+\rho_M(Y\cup\{x\})$.

For the reverse inequality, apply submodularity to $xE$ and
$y(Y\cup\{x\})$. Their union is $xE\cup yY$, since $xy\in xE$,
and their intersection is $\{xy\}$.
The singleton $\{xy\}$ has rank $\rho_M(\{y\})=1$ by the
isomorphism $M\cong N|_{xE}$. Thus
\[
4+\rho_M(Y\cup\{x\})\ge\rho_N(xE\cup yY)+1.
\]
This proves the auxiliary identity.

We now prove \eqref{eq:rows}.
The sets $xY\cup yY$ and $xE$ satisfy
\[
(xY\cup yY)\cup xE=xE\cup yY,\qquad
xY\subseteq(xY\cup yY)\cap xE.
\]
By monotonicity and \eqref{eq:defpower}, the intersection has rank
at least $\rho_M(Y)$. Submodularity consequently gives
\[
\rho_N(xY\cup yY)+4\ge\rho_N(xE\cup yY)+\rho_M(Y).
\]
Substituting the equality already proved yields
\[
\rho_N(xY\cup yY)\ge
3+\rho_M(Y\cup\{x\})+\rho_M(Y)-4
=\rho_M(Y)+\rho_M(Y\cup\{x\})-1,
\]
as required.
\end{proof}

\begin{lemma}\label{lem:products}
We have
\[
\begin{aligned}
\rho_N(A(B\cup C))&=6,&
\rho_N(C(A\cup B))&=6,\\
\rho_N(BD)&=4,&
\rho_N(A(A\cup D))&=5.
\end{aligned}
\]
\end{lemma}
\begin{proof}
First, write $A=\{x,y\}$ and put $Y=B\cup C$, so that
$A(B\cup C)=xY\cup yY$.
Since $Y\in\mathcal H$ is a rank-three hyperplane and $x\notin Y$,
\[
\rho_M(Y)=3,\qquad \rho_M(Y\cup\{x\})=4.
\]
The inequality in \eqref{eq:rows} gives
$\rho_N(A(B\cup C))\ge3+4-1=6$.
On the other hand, \eqref{eq:defpower} gives
$\rho_N(xY)=\rho_N(yY)=3$.
Submodularity implies subadditivity of rank, so
\[
\rho_N(xY\cup yY)\le\rho_N(xY)+\rho_N(yY)=6.
\]
Thus $\rho_N(A(B\cup C))=6$.

Next, write $C=\{x,y\}$ and put $Y=A\cup B$.
Again, $Y$ is a rank-three hyperplane and $x\notin Y$, so
$\rho_M(Y)=3$ and $\rho_M(Y\cup\{x\})=4$.
The inequality in \eqref{eq:rows} gives
$\rho_N(C(A\cup B))\ge6$.
Each of $xY$ and $yY$ has rank three, so subadditivity gives the
reverse inequality. Hence $\rho_N(C(A\cup B))=6$.

For $BD$, write $D=\{x,y\}$ and take $Y=B$.
The set $B$ is independent of rank two, and $B\cup\{x\}$ is independent
of rank three. Thus \eqref{eq:rows} and subadditivity give
\[
4=2+3-1
\le\rho_N(xB\cup yB)
\le\rho_N(xB)+\rho_N(yB)=4.
\]
Since $BD=xB\cup yB$, its rank is four.

Finally, write $A=\{x,y\}$ and put $Y=A\cup D$.
We have $\rho_M(Y)=3$ and $x\in Y$, so the inequality in
\eqref{eq:rows} gives
$\rho_N(A(A\cup D))\ge3+3-1=5$.
For the upper bound, note that $\rho_N(xY)=\rho_N(yY)=3$ and
$xY\cap yY=\{xy\}$, since $x,y\in Y$.
This singleton has rank one. Submodularity therefore yields
\[
\rho_N(A(A\cup D))
=\rho_N(xY\cup yY)
\le3+3-1=5.
\]
Combining the two bounds proves the last equality.
\end{proof}

\subsection{Proof of the main theorem}\label{sec:main-proof-core}
We use the notation introduced above.
The proof is based on the following subset of $\Sym^2(E)$:
\[
X:=AB\cup AC\cup\Sym^2(D).
\]

\begin{proof}[Proof of Theorem~\ref{thm:main}]
We will obtain a contradiction by proving that
\[
\rho_N(X)\ge8
\qquad\text{and}\qquad
\rho_N(X)\le7.
\]

\medskip\noindent
\textbf{Step 1: Two twelve-element sets $T$ and $U$ have rank at most seven.}\\
Put
\[
T:=AB\cup AC\cup BC,\qquad U:=AB\cup AC\cup BD.
\]
Consider the three intersections
\begin{align*}
I_1&:=F_{A\cup B}\cap F_{A\cup C}=F_A\cup BC,\\
I_2&:=F_{A\cup B}\cap F_{B\cup C}=F_B\cup AC,\\
I_3&:=F_{A\cup B}\cap F_{A\cup D}=F_A\cup BD.
\end{align*}
For example, a pair lies in both $F_{A\cup B}$ and $F_{A\cup C}$
precisely when at least one of its endpoints lies in $A$, or one endpoint
lies in $B$ and the other in $C$.
This proves the identity for $I_1$; the other two identities follow
in the same way.

The sets $A\cup B$, $A\cup C$, $B\cup C$, and $A\cup D$ all belong
to $\mathcal H$ and have rank three.
By Lemma~\ref{lem:ranks},
\[
\rho_N(F_{A\cup B})=\rho_N(F_{A\cup C})
=\rho_N(F_{B\cup C})=\rho_N(F_{A\cup D})=9.
\]
The sets $A\cup B\cup C$ and $A\cup B\cup D$ have rank four, so
\[
\rho_N(F_{A\cup B\cup C})
=\rho_N(F_{A\cup B\cup D})=10.
\]
Since $F_{A\cup B}\cup F_{A\cup C}=F_{A\cup B\cup C}$,
submodularity gives
\[
9+9\ge10+\rho_N(I_1),
\]
and hence $\rho_N(I_1)\le8$.
The identities
\[
F_{A\cup B}\cup F_{B\cup C}=F_{A\cup B\cup C},
\qquad
F_{A\cup B}\cup F_{A\cup D}=F_{A\cup B\cup D}
\]
give the same bound for $I_2$ and $I_3$. Thus
\[
\rho_N(I_j)\le8\qquad(j=1,2,3).
\]

To bound the rank of $T$, observe that
\[
I_1\cup I_2=F_{A\cup B},\qquad I_1\cap I_2=T.
\]
Applying submodularity to $I_1$ and $I_2$ gives
\[
\rho_N(I_1)+\rho_N(I_2)\ge9+\rho_N(T),
\]
so $\rho_N(T)\le8+8-9=7$.
Similarly,
\[
I_2\cup I_3=F_{A\cup B},\qquad I_2\cap I_3=U
\]
imply $\rho_N(U)\le7$. We have therefore proved
\begin{equation}
\rho_N(T)\le7,\qquad \rho_N(U)\le7.
\label{eq:TU}
\end{equation}

\medskip\noindent
\textbf{Step 2: An eleven-element set $X$ has rank at least eight.}\\
Lemmas~\ref{lem:products} and~\ref{lem:ranks} give
\[
\rho_N(AC\cup BC)=\rho_N(C(A\cup B))=6,\qquad
\rho_N(F_C)=7.
\]
Since $C\cup D$ is a basis of $M$, Lemma~\ref{lem:ranks} also gives
$\rho_N(\Sym^2(C\cup D))=10$.
Every pair in $\Sym^2(C\cup D)$ either has an endpoint in $C$
or has both endpoints in $D$. Consequently,
\[
\Sym^2(C\cup D)\subseteq F_C\cup\Sym^2(D).
\]
Since $N$ has rank ten, monotonicity gives
\[
10\le\rho_N(F_C\cup\Sym^2(D))\le10,
\]
so $\rho_N(F_C\cup\Sym^2(D))=10$.

We have $AC\cup BC\subseteq F_C$ and
$F_C\cap\Sym^2(D)=\varnothing$, since $C\cap D=\varnothing$.
Hence
\begin{align*}
F_C\cap\bigl((AC\cup BC)\cup\Sym^2(D)\bigr)&=AC\cup BC,\\
F_C\cup\bigl((AC\cup BC)\cup\Sym^2(D)\bigr)&=F_C\cup\Sym^2(D).
\end{align*}
Applying submodularity to $F_C$ and $(AC\cup BC)\cup\Sym^2(D)$ gives
\[
7+\rho_N\bigl((AC\cup BC)\cup\Sym^2(D)\bigr)\ge10+6.
\]
Thus $\rho_N((AC\cup BC)\cup\Sym^2(D))\ge9$.
As $AC\cup BC\subseteq T$, monotonicity yields
\begin{equation}
\rho_N(T\cup\Sym^2(D))\ge9.
\label{eq:TSD-lower}
\end{equation}

Because the four blocks are pairwise disjoint,
\[
X\cap T=AB\cup AC=A(B\cup C),\qquad
X\cup T=T\cup\Sym^2(D).
\]
Lemma~\ref{lem:products} gives $\rho_N(X\cap T)=6$.
Submodularity, together with \eqref{eq:TU} and \eqref{eq:TSD-lower},
now gives
\begin{equation}
\begin{aligned}
\rho_N(X)
&\ge\rho_N(X\cup T)+\rho_N(X\cap T)-\rho_N(T)\\
&\ge9+6-7=8.
\end{aligned}
\label{eq:lower}
\end{equation}

\medskip\noindent
\textbf{Step 3: The same set $X$ has rank at most seven.}\\
We first bound the rank of $F_A\cup\Sym^2(D)$.
Since $\Sym^2(A)$ and $AD$ are contained in $F_A$,
\[
\begin{aligned}
F_A\cup\Sym^2(A\cup D)&=F_A\cup\Sym^2(D),\\
F_A\cap\Sym^2(A\cup D)&=A(A\cup D).
\end{aligned}
\]
As $\rho_M(A)=2$ and $\rho_M(A\cup D)=3$,
Lemma~\ref{lem:ranks} gives
\[
\rho_N(F_A)=7,\qquad \rho_N(\Sym^2(A\cup D))=6.
\]
Lemma~\ref{lem:products} gives $\rho_N(A(A\cup D))=5$.
Submodularity therefore implies
$7+6\ge\rho_N(F_A\cup\Sym^2(D))+5$, and hence
\begin{equation}
\rho_N(F_A\cup\Sym^2(D))\le8.
\label{eq:FA-SD-upper}
\end{equation}

We next bound the rank of $U\cup\Sym^2(B\cup D)$.
Since $B\cup D\in\mathcal H$, Lemma~\ref{lem:ranks} gives
$\rho_N(\Sym^2(B\cup D))=6$.
Moreover,
\[
U\cap\Sym^2(B\cup D)=BD.
\]
Indeed, among the pairs in $U=AB\cup AC\cup BD$, precisely those
in $BD$ have both endpoints in $B\cup D$.
By Lemma~\ref{lem:products}, $\rho_N(BD)=4$.
Submodularity and \eqref{eq:TU} give
\[
\rho_N\bigl(U\cup\Sym^2(B\cup D)\bigr)
\le\rho_N(U)+6-4\le7+2,
\]
so
\begin{equation}
\rho_N\bigl(U\cup\Sym^2(B\cup D)\bigr)\le9.
\label{eq:U-SBD-upper}
\end{equation}

We now apply submodularity to $F_A\cup\Sym^2(D)$ and
$U\cup\Sym^2(B\cup D)$. We first verify that their intersection is \(X\). Every pair in $AB\cup AC$ belongs to $F_A$.
On the other hand, every pair in $BD$ or $\Sym^2(B\cup D)$
has both endpoints in $B\cup D$.
Since $A\cap(B\cup D)=\varnothing$, none of these pairs belongs
to $F_A$. Therefore
\[
F_A\cap\bigl(U\cup\Sym^2(B\cup D)\bigr)=AB\cup AC.
\]
Since $\Sym^2(D)\subseteq\Sym^2(B\cup D)$, it follows that
\[
\bigl(F_A\cup\Sym^2(D)\bigr)
\cap\bigl(U\cup\Sym^2(B\cup D)\bigr)
=AB\cup AC\cup\Sym^2(D)=X.
\]
Their union is $F_A\cup\Sym^2(B\cup D)$, because
$AB\cup AC\subseteq F_A$ and
$BD\cup\Sym^2(D)\subseteq\Sym^2(B\cup D)$.
Moreover,
\[
\Sym^2(A\cup B\cup D)\subseteq F_A\cup\Sym^2(B\cup D).
\]
To see this, a pair with both endpoints in $A\cup B\cup D$
either has an endpoint in $A$, in which case it belongs to $F_A$,
or has both endpoints in $B\cup D$, in which case it belongs to
$\Sym^2(B\cup D)$.
Since $A\cup B\cup D$ has rank four, Lemma~\ref{lem:ranks} gives
$\rho_N(\Sym^2(A\cup B\cup D))=10$.
Since $N$ has rank ten, the preceding inclusion implies
\[
\rho_N\bigl(F_A\cup\Sym^2(B\cup D)\bigr)=10.
\]

Finally, submodularity, \eqref{eq:FA-SD-upper}, and
\eqref{eq:U-SBD-upper} yield
\[
\begin{aligned}
\rho_N(X)
&\le\rho_N\bigl(F_A\cup\Sym^2(D)\bigr)
   +\rho_N\bigl(U\cup\Sym^2(B\cup D)\bigr)\\
&\quad-\rho_N\bigl(F_A\cup\Sym^2(B\cup D)\bigr)\\
&\le8+9-10=7.
\end{aligned}
\]
Together with \eqref{eq:lower}, this gives
$8\le\rho_N(X)\le7$, a contradiction.
Hence $V_8$ has no second symmetric power.
\end{proof}

\noindent
\textsc{Yasuhito Nakajima}\\
Independent Scholar\\
Email: \texttt{yasuhito.nakajima.mathematics@gmail.com}
\end{document}